\documentclass[a4paper, 12pt]{amsart}

\usepackage{amssymb}
\usepackage{amsmath}
\usepackage{amsthm}
\usepackage{amscd}
\usepackage{comment}

\title{Hausdorff dimension of asymptotic self-similar sets}

\author{Daruhan Wu}
\address{Graduate School of Pure and Applied Sciences, University of Tsukuba, 
  Tsukuba, 305-8571, Japan}
\curraddr{School of Statistics and Mathematics, Inner Mongolia University of Finance and Economics, Huhhot, China}
 \email{daruhan@imufe.edu.cn}

\author{Takao Yamaguchi*}
\email{takaoy@math.kyoto-u.ac.jp}
\address{Department of mathematics, Kyoto University, Kitashirakawa, Kyoto 606--8502, Japan}

\thanks{This work was supported by JSPS KAKENHI Grant Numbers  26287010, 15K13436, 15H05739}

\date{\today}
\theoremstyle{plain}
\newtheorem{theorem}{Theorem}[section]
\newtheorem{lemma}[theorem]{Lemma}
\newtheorem{corollary}[theorem]{Corollary}

\newtheorem{sublemma}[theorem]{Sublemma}

\newtheorem{definition}[theorem]{Definition}

\newtheorem{assertion}[theorem]{Assertion}
\newtheorem{example}[theorem]{Example}

\makeatletter

\newcommand{\supp}[0]{\mathrm{supp}}

\begin{document}
\begin{abstract}
In this paper, we introduce the notion of asymptotic self-similar sets on general 
doubling metric spaces by extending the notion of  self-similar sets, 
and determine their Hausdorff dimensions, which gives an extension of  Balogh and Rohner 's result. 
This is carried out by introducing the notions of almost similarity maps and asymptotic similarity systems. 
These notions  have an advantage of
making geometric constructions possible.
Actually, as an application, we determined the Hausdorff dimension of general Sierpinski gaskets on complete surfaces
constructed by a geometric way in a natural manner.
\end{abstract}

%\end{abstract}
\maketitle

%%%%%%%%%start section {introduction}

\section{Introduction}
%
%%% background %%%%%%%%%%%%%%%%%%%%%%%%%%%%%%%%%%%%%%%%%%%%%%%%%%%%%%%%%%%%%%%%%%%%%%%%%
The notion of self-similar sets or general Cantor sets have played significant roles in fractal geometry. 
These sets are usually  defined by means of iterated function systems 
$\{f_1,\cdots,f_k\}$ consisting of contracting similarity maps on a complete metric space 
as the unique nonempty compact set $K$,  called  an attractor or an invariant set, 
satisfying $K={\bigcup}_{i=1}^n f_i(K)$.
Hutchinson \cite{Hutchinson} (cf. Kigami \cite{Kigami}, Schief \cite{Schief}) introduced the notion of 
the open set condition and determined the Hausdorff dimension of  self-similar sets in Euclidean space $\mathbb{R}^n$ 
satisfying the open set condition. %as the similarity dimension.
Balogh and Rohner extended Hutchinson's result to doubling metric spaces (\cite{Balogh}).
However, it is difficult to construct a similarity map in general metric spaces. \ Actually, similarity maps 
do not always exist on curved metric spaces. 
To overcome this difficulty, in the previous work \cite{Wu2}, the first named author  introduced 
the notion  of $(\lambda, c,\nu)$-almost similarity maps extending that of 
$\lambda$-similarity maps in order to construct 
generalized Cantor sets in general metric measure spaces,  and determined the Hausdorff dimension
of such a generalized Cantor set. 
However the basic subsets considered in \cite{Wu2}  are assumed to be  disjoint each other, and therefore
generalized Cantor sets like Sierpinski gaskets are excluded in the results of \cite{Wu2}.

In the present paper, we extend both Balogh and Rohner 's result and our previous result
to the case when basic subsets may have intersections with their boundary 
by introducing a generalized open set condition. As an application, we determine the Hausdorff 
dimension of Sierpinski gaskets on complete surfaces defined via  geometric way.  

%%%%%%%%  doubling  metric space %%%%%%%%%%%%%%%%%%%%%%%%%%%%%%%%%%%%%%%%%%%%%%%%%%%%%%%%%%%%%%%%%%%%%%
Let $X$ be a proper complete metric space. We assume that $X$
is doubling in the sense  of \cite{Balogh} (see Section \ref{sec:prelim} for the precise definition).
Complete Riemannian manifolds with Ricci curvature bounded from below are 
typical examples of doubling metric spaces (cf. \cite{Grom}).
Doubling metric spaces also appears in metric measure spaces satisfying a doubling condition. 
%The assumption $(b)$ of our Main Theorem is essentially the same as considering 
% 
Nowadays,  geomeric analysis on doubling metric measure spaces has been very active 
(see for instance Assouad \cite{Assouad}, Gromov\cite{Grom}, Heinonen \cite{Heinonen}, 
Villani\cite{Vill}), and therefore it is quite natural to study self-similarity sets 
in  such doubling metric spaces.   
%%%%%%%%%%%%%%%%%%%  almost similarity map %%%%%%%%%%%%%%%%%%%%%%%%%%%%%%%%%%%%%%%%%%%%%%%%%%%%%%%%%%%

Let $\bar U \supset \bar V$ be  bounded domains  in  $X$ homeomorphic to each other,
where $\bar U$ and $\bar V$ denote the closures of the open subsets 
$U$ and $V$.
Fix  constants $0<\lambda<1$, $0<\nu<1$ and a continuous increasing  function 
$\varphi:(0,\infty)\to (0,\infty)$ with 
$\lim_{x\to +0}\varphi(x)=0$.
We call a homeomorphism $f:\bar U \to \bar V$ a {\it $(\lambda, \varphi(|\bar U|),\nu)$-almost similarity map} 
if  for every $x,y\in \bar U$,
\begin{align}
   \left|\frac{d(f(x),f(y))}{d(x,y)}-\lambda\right|&\leq \lambda\varphi(|U|),\\
   |V|\le  \nu |U|.
\end{align}
where\ $|U|$ is the diameter of $U$.
Then the set $\bar V$ is called a $(\lambda, \varphi(|\bar U|),\nu)$-almost similar  set of  $\bar U$.\\

%%%%%%%%%%%%%%%%%%%   $\varphi$ and \mathcal I  %%%%%%%%%%%%%%%%%%%%%%%%%%%%%%%%%%%%%%%%%%%%%%%%%%%%%%%%%%
In this paper, we assume the following conditions for $\varphi$:
\begin{equation}
 \begin{aligned}
    & \varphi:(0,\infty)\to (0,\infty) \,\,\rm{is\,\, increasing \, \,with} \,\,\lim_{x\to +0} \varphi(x)=0; \\
   &  \displaystyle{ \int_1^{\infty} \varphi(a\nu^{x})\,dx <\infty} \,\,\rm{for\,\, some\,\, constants} \,\,a>0 \,\, \rm{and}\,\, 0<\nu<1. \label{eq:varphi}
\end{aligned}
\end{equation}
%merate}
%  \item  $\varphi:(0,\infty)\to [0,\infty)$ is non-decreasing  with $\lim_{x\to +0} \varphi(x)=0;$
%  \item  $\displaystyle{ \int_1^{\infty} \varphi(a\nu^{x})\,dx <\infty}$ for some constants $a>0$ and $0<\nu<1$.
%\end{enumerate}
Note that the second condition $(2)$ above does not depend on the choice of 
$a>0$ and $0<\nu<1$, and that for any $\alpha>0$ and any positive integer $n$, 
the following functions satisfy the above conditions:
\[
   \varphi(y)=y^{\alpha},   \,\, \,  \varphi(y) = - (\log y)^{-1-\frac{2}{2n+1}}. 
\]
For a fixed positive integer $k$, we let $\mathcal I=\{ 1,2,\ldots,k\}$. 
We denote by $\mathcal I^*$ the set of all ordered multi-indices
$I=i_1\cdots i_n$ with $n\ge 1$, $i_j\in\mathcal I$ for every $1\le j \le n$.
We set $|I|=|i_1\cdots i_n|=n$ and call it the length of $I$. 
Let $\mathcal I^n$ denote the set of all $I\in\mathcal  I$ of length $n$.

%%%%%%%%%%%%%%%%%%%  DEF(asymptotic self-similar set ) %%%%%%%%%%%%%%%%%%%%%%%%%%%%%%%%%%%%%%%%%%%%%%%%%%%%%%%%%%%
In the present paper, we investigate an asymptotic self-similar set  in $X$, which is 
defined under the following hypothesis:
For $0<\nu<1$ and $a>0$, let $\varphi:(0,\infty)\to (0,\infty)$ be a continuous function satisfying 
the above conditions \eqref{eq:varphi}. 

\begin{definition} \upshape
Suppose that 
 {\it ratio coefficients} $0<\lambda_i<1$, $ (i=1,\ldots,k)$
together with a non-empty open subset $V\subset X$ 
are given for which we have 
\begin{enumerate}
\item for each  $i \in\mathcal I$,  a $(\lambda_i, \varphi(|\bar V|),\nu)$-almost similarity map 
\[
                        f_i:\bar{V}\to\bar V_i\subset \bar{V},
\]
is  given in such a way that
$V_i\cap V_j=\emptyset$ for every $i\neq j\in\mathcal I$,  where $V_i:=f_i(V);$
\item for each $ij\in\mathcal I^2$,
a $(\lambda_j, \varphi(|\bar V_{i}|),\nu)$-almost similarity map 
\[
                     f_{ij}:\bar{V_{i}}\to  \bar V_{ij}\subset \bar{V}_{i},
\] 
is given in such a way that ${V}_{ij}\cap {V}_{ij^\prime}=\emptyset$
for every $j\neq j'\in\mathcal I$,  where ${V}_{ij}:=f_{ij}({V}_{i})$; 
\item for each $I'\in\mathcal I^{n-1}$ and $i_n\in\mathcal I$ with $I:=I' i_n$,
a $(\lambda_{i_{n}},\varphi(|\bar V_{I'}|),\nu)$-almost similarity map
\[
       f_{I}:\bar{V}_{I'} \to\bar{V}_{I}\subset \bar V_{I'},
\]
is defined 
in such a way that $V_{I'i} \cap V_{I'j}=\emptyset$ for every $i\neq j\in\mathcal I$,
where ${V}_{I}:=f_{I}(V_{I'})$.   
\end{enumerate}

We call $\{ (\bar V_I, f_I)\}_{I\in \mathcal I^*}$ an {\it $(\{ \lambda_i\}_{i=1}^k, \varphi,\nu)$-asymptotic similarity system}.
Then the  set $K$  defined as
\begin{equation*}
   K=\displaystyle\bigcap_{n=1}^{\infty}\left(\displaystyle\bigcup_{I\in\mathcal I^n}  \bar{V}_{I}\right),
\end{equation*}
is called an {\it asymptotic self-similar set} in $X$.\\
\end{definition}

Let us consider the case of iterated function system of contracting 
similarity maps   $\{ f_1,\ldots, f_k \}$ with open set condition
\begin{enumerate}
 \item $V\supset f_1(V)\cup \cdots \cup f_k(V);$
 \item  $f_i(V)\cap f_j(V)\neq\emptyset$ for every $i\neq j;$
\end{enumerate}
for some non-empty open set $V\subset X$. In this case, 
for each $I=i_1\cdots i_n\in\mathcal I^*$, let
\[
      V_I:=f_{i_n}\circ\cdots\circ f_{i_1}(V), \,\,\,   f_I:=f_{i_n}:\bar V_{I'}\to \bar V_{I}. 
\]
Then this gives a  $(\{ \lambda_i\}_{i=1}^k, \varphi=0,\lambda_{\max})$-asymptotic similarity system  
$\{ (\bar V_I, f_I)\}_{I\in \mathcal I}$,
where $\lambda_{\max}=\max \lambda_i$.
Thus the notion of $(\{ \lambda_i\}_{i=1}^k, \varphi,\nu)$-asymptotic similarity system is an 
extension of  iterated function system of contracting similarity maps with open set condition.

%%%%%%%%%%% Main theorem %%%%%%%%%%%%%%%%%%%%%%%%%%%%%%%%%%%%%%%%%%%%%%%%%%%%%%%%%%%%%%%%%%%%%%%%%%%

Our main result in the present paper is stated as follows.

\begin{theorem}\label{T1}
Let $X$ be a complete doubling metric space %with a regular Borel measure $\mu$, 
and let $K$ be the asymptotic self-similar set associated with a 
$(\{ \lambda_i\}_{i=1}^k, \varphi,\nu)$-asymptotic similarity system $\{ (\bar V_I, f_I)\}_{I\in\mathcal I}$.
Then the Hausdorff and the box dimensions of $K$ are given as
\[
       \dim_H K=\dim_B K =s,
\]
where $s$ is a unique number satisfying  
$\displaystyle \sum_{i=1}^k \lambda_i^s =1$.
\end{theorem}
%%%%%%%%%%%%%%%%  Meaning  of Main theorem  %%%%%%%%%%%%%%%%%%%%%%%%%%%%%%%%%%%%%%%%%%

In \cite{Balogh}, Balogh and Rohner suggested a problem.
They considered an iterated function system 
of contracting 
{\it asymptotically similarity maps} in the sense that for all $I=i_i\cdots i_n\in\mathcal I$
\[
       c_1\lambda_I \le \frac{|f_I(x), f_I(y)|}{|x,y|}\le c_2\lambda_I,
\]
where $f_I=f_{i_n}\circ\cdots\circ f_{i_1}$,  $\lambda_I=\lambda_{i_1}\cdots\lambda_{i_n}$
 and $c_1$,  $c_2$ are uniform positive constants.
They posed a problem: What happens if an iterated function system 
of contracting similarity maps is replaced by one of contracting 
asymptotically similarity maps ?
Rajala and Vilppolainen completely solved the above problem in Theorem 4.9 of  \cite{RV}
by introducing a more general notion of  {\it a semiconformal iterated function system}.
A $(\{ \lambda_i\}_{i=1}^k, \varphi,\nu)$-asymptotic similarity system  $\{ (\bar V_I, f_I)\}_{I\in \mathcal I}$ 
is closely related with Balogh and Rohner's  iterated function system of contracting asymptotically 
similarity maps and Rajala and Vilppolainen's  semiconformal iterated function system
under the open set condition. 
Actually our notion of asymptotic similarity 
system  provides a controlled Moran construction in the 
sense of Rajala and Vilppolainen (\cite{RV}) (see Lemma \ref{lem:remarkRV}).
However  an asymptotic self-similar set introduced in the present paper is 
constructed by means of a $(\{ \lambda_i\}_{i=1}^k, \varphi,\nu)$-asymptotic similarity system,
which consists of infinite series of almost similarity maps.
Therefore in general, it is not simply  defined by a finite iterated function system.
For example, a generalized Sierpinski gasket on a general complete surfaces 
constructed in this paper is an asymptotic self-similar set. 
It would be an interesting question to determine whether  a generalized Sierpinski gasket on a general complete surface
can be defined by means of  a finite iterated function system due to 
 Balogh-Rohner or Rajala-Vilppolainen (see Section \ref{sec:Sierp}).
Anyway the notion of asymptotic self-similar sets introduce in this paper has an advantage of
making geometric constructions in general curved spaces  much easier.

%%%%%%%%%%%%%%%%  Sierpinski gasket  %%%%%%%%%%%%%%%%%%%%%%%%%%%%%%%%%%%%%%%%%%
As indicated above,  we consider a Sierpinski gasket  on 
a complete surface $M$ as an application of Theorem \ref{T1}, which is naturally defined in a geometric way as follows.

Now let $\mathcal I=\{ 1, 2, 3\}$, and
let $\Delta$ be a closed domain contained in a convex domain of $M$ bounded by a geodesic triangle. 
By joining the midpoints of the edges of $\Delta$ by minimal geodesics, we divide $\Delta$ into 
four triangles, and remove the center triangle to get three geodesic triangles $\Delta_1$, $\Delta_2$ and
$\Delta_3$. Repeating this procedure for each $\Delta_i$ infinitely many times, 
we obtain a system of geodesic triangles $\{ \Delta_I\}_{I\in\mathcal I^*}$.
	The {\it generalized Sierpinski gasket} $K_{\Delta}$ on $M$  associated with $\Delta$ is defined as
\begin{equation*}
   K_{\Delta}=\displaystyle\bigcap_{n=1}^{\infty}\left(\displaystyle\bigcup_{I\in\mathcal I^n}  \Delta_{I}\right),
\end{equation*}
We say that  $\Delta$ is {\it asymptotically non-degenerate} if all the divided small triangles $\Delta_I$
are $\delta$-non-degenerate for some constant $\delta>0$. (See Section \ref{sec:Sierp}
for the precise definition). For example, every geodesic triangle region $\Delta$ of perimeter 
less than $2\pi$ on a unit sphere is  asymptotically non-degenerate (see Example \ref{ex:sphere}).
We show that a small  geodesic  triangle region on a surface is 
 asymptotically non-degenerate  (see Lemma \ref{lem:nondeg2}).

\begin{theorem} \label{thm:Sierp}
If a geodesic triangle domain $\Delta$ in a convex domain on a complete surface is 
asymptotically non-degenerate, then  
\begin{enumerate}
 \item   for some  $0<\nu<1$ there exists a $(\{ 1/2,1/2,1/2\},\varphi, \nu)$-asymptotic similarity system
 $\{ (\Delta_I,f_I)\}_{I\in \mathcal I^*}$ associated with $\Delta$, where  $\varphi(x)=cx^2$ 
 for some constant $c>0;$
 \item the Hausdorff and box dimensions 
of the generalized Sierpinski gasket $K_{\Delta}$ associated with $\Delta$ 
are  given by
   \begin{equation}
       \dim_H K_{\Delta} = \dim_B K_{\Delta} = \frac{\log3}{\log2}. \label{eq:dimHB}
   \end{equation} 
\end{enumerate}

\end{theorem}

The following result gives a condition for $\Delta$ to be asymptotically non-degenerate.

\begin{corollary}\label{thm:nondeg}
A geodesic triangle domain $\Delta$ in a convex domain on a complete surface is 
asymptotically non-degenerate if and only if for some  $0<\nu<1$
there exists a $(\{ 1/2,1/2,1/2\},\varphi, \nu)$-asymptotic similarity system
 $\{ (\Delta_I,f_I)\}_{I\in \mathcal I^*}$ associated with $\Delta$, where  $\varphi(x)=cx^2$ 
 for some constant $c>0$.
\end{corollary}
 
The organization of the present paper is as follows:
In Section \ref{sec:prelim}, we discuss some basic notions needed in the proof of the above results.
In Section \ref{sec:main}, we prove Theorem \ref{T1}.
In Section \ref{sec:Sierp}, we discuss generalized Sierpinski gaskets on complete surfaces, 
and prove Theorem \ref{thm:Sierp} and Corollary \ref{thm:nondeg}.

%%%%%%%%%%%%%%%%%%%% Acknowlegement %%%%%%%%%%%%%%%%%%%%%%%%%%%%%%%%%%%%%%%
The authors would like to thank Ayato Mitsuishi for a comment on Example \ref{ex:sphere}.
We would also like to thank the referee for valuable advice.

%%%%%%%%%%%%%%%%%  \section{preliminaries}  %%%%%%%%%%%%%%%%%%%%%%%%%%%%%%%%%%%%%%%%%%
\section{preliminaries}\label{sec:prelim}
%%%%%% doubling %%%%%%%%%

The distance between points $x,y$ in a metric space will be denoted as  $d(x,y)$.
 For $r>0$, $B(x,r)$ denotes the open ball of radius $r$ around $x$.

\begin{definition} \upshape
A metric space $X$ is said to be {\it doubling} if 
there exists a positive integer $C$ such that for any $x\in X$ and any $r>0$,
there exist $\{ x_i\}_{i=1}^{C} \subset X$ such that 
\[
B(x,r)\subset \displaystyle\bigcup_{i=1}^{C}B(x_i, r/2)
\]
Note that $C$, called the {\it doubling constant} of $X$, does not dependent on the choices of $x$ or $r$.
\end{definition}

For the proof  of the following lemma, see   Lemma 3.3 of \cite{Balogh}.

\begin{lemma} \label{lem:doubling}
 Let $X$ be a doubling metric space with doubling constant $C$.
For any $0<\delta<1$, t�here exists a constant $C(\delta)$ such that 
the number of mutually disjoint balls $B(x_i, \delta r)$ in a ball $B(x, r)$ of $X$
is bounded by $C(\delta)$.
\end{lemma}

%%%%%% H-measure %%%%%%%%%
\begin{definition} \upshape
Let $X$ be a metric space,\ $A\subset X$ and $\alpha$ be a nonnegative real number.
An {\it $\epsilon$-cover} $\{U_i \}$ of $A$ is a finite or countable collection of sets $U_i$  covering $A$ with $|U_i|\le\epsilon$.
Define $\mathcal{H}_{\epsilon}^{\alpha}(A)$  by
\begin{equation*}
\mathcal{H}_{\epsilon}^{\alpha}(A)=\inf\Big\{\displaystyle\sum_{i=1}^{\infty}|U_i|^{\alpha}\ \big|\ 
    \{U_i \}:\ \epsilon\text{-cover of} \, \,A\Big\}.
\end{equation*}
The {\it $\alpha$-dimensional Hausdorff measure of} $A$ is defined by  
\begin{equation*}
\mathcal{H}^{\alpha}(A)=\displaystyle \lim_{\epsilon \to 0}\mathcal{H}_{\epsilon}^{\alpha}(A),
\end{equation*}
and the {\it Hausdorff dimension $\dim_H A$ of} $A$ is defined as
\begin{equation*}
\dim_H  A:={\rm sup} \{\alpha\ge 0 |\mathcal{H}^{\alpha}(A)=\infty\}={\rm inf} \{\alpha\ge 0 |\mathcal{H}^{\alpha}(A)=0 \}.
\end{equation*}
\end{definition}
%%%%%% B-dimension %%%%%%%%%
Let $A$ be a bounded subset of a metric space $X$.
Let $N_{\epsilon}(A)$ denote the minimal number of subsets of 
diameter $\le \epsilon$ needed to cover $A$.
The {\it  lower box dimension} and the {\it upper box dimension} of $A$
are defined respectively as 
\begin{align*}
 \underline{\dim}_B A  =  \mathop{\underline{\lim}}_{\epsilon\to 0} \frac{\log N_{\epsilon}(A)}{-\log\epsilon}, \,\,\,
 \overline{\dim}_B A  =  \mathop{\overline\lim}_{\epsilon\to 0} \frac{\log N_{\epsilon}(A)}{-\log\epsilon}.
\end{align*}
When both the lower and the upper box dimensions are equal, the common value 
\begin{align*}
 \dim_B A =  \lim_{\epsilon\to 0} \frac{\log N_{\epsilon}(A)}{-\log\epsilon}
\end{align*}
is called the {\it box dimension} of $A$.

The following is a standard fact (see \cite{Fal2} for instance):

%\begin{lemma} \label{lem:box}
 \begin{align}
    \dim_H A \le  \underline{\dim}_B A \le \overline{\dim}_B A. \label{eq:dims}
\end{align}
%\end{lemma}

%%%%%% self-similar measure %%%%%%%%%

Next we  discuss self-similarity measures.
In the rest of this section,  we always assume that $Y$ is a compact metric space 
unless otherwise stated.

Let $\mathcal M(Y)$ be the set of all Borel probability measures on $Y$.
Consider the {\it  Kantrovich-Rubinshtein metric} $d_{\mathcal M}$ and 
 the {\it modified Kantrovich-Rubinshtein metric} $d_{\mathcal M}^*$ on $\mathcal M(Y)$ 
defined by
\begin{align*}
   d_{\mathcal M}(\mu_1,\mu_2) 
      =& \sup \left\{ \left| \int_{Y}\phi \,d\mu_1 -  \int_{Y}\phi \,d\mu_2 \right| :
     \phi\in\mathrm{Lip}_1(Y), \,\sup_{x\in Y} |\phi(x)|\le 1    \right\}, \\
   d_{\mathcal M}^*(\mu_1,\mu_2) 
      = & \sup \left\{ \left| \int_{Y}\phi \,d\mu_1 -  \int_{Y}\phi \,d\mu_2 \right| : 
     \phi\in\mathrm{Lip}_1(Y)    \right\},\\
\end{align*}
where $\mathrm{Lip}_1(Y)$ denotes the set of all Lipschitz functions on $Y$ 
with Lipschitz constant $\le 1$.

It is well known that  $(\mathcal M(Y), d_{\mathcal M})$
is complete (see Theorem 8.10.43 of \cite{Bog}). Further,  we have from the definition
\[
      d_{\mathcal M}(\mu_1,\mu_2)\le  d_{\mathcal M}^*(\mu_1,\mu_2)
                      \le \max\{ |Y|,1\}  d_{\mathcal M}(\mu_1,\mu_2).
\]
In particular,  $(\mathcal M(Y), d_{\mathcal M}^*)$ is also complete.
  
%By Riesz's representation formula, $\mathcal M(X)$ is a complete metric space.

Let $\{f_i\}_{i=1}^m $ be a family of contracting maps in a compact metric space $Y$.
Namely, there are some constants $0<\lambda_i<1$ such that
\[
        \frac{d(f_i(x), f_i(y))}{d(x,y)}\le \lambda_i <1,
\]
 for every $x\neq y\in Y$ and $1\le i\le m$. 
%
%Then 
%there exists a compact subset $K$ of $X$ such that  
%\[
%     K=f_1(K)\cup\cdots\cup f_m(K).
%\]
%The set $K$ is called the  {\it invariant  set} of  $\{ f_i \}_{i=1}^m$
%in general.
%  
\begin{lemma}(cf. \cite{KT})  \label{lem:similar-measure} 
Let $Y$ and  $\{f_i\}_{i=1}^m $ be as above. 
 Then  for  any positive numbers
   $a_i$, $1\le i\le m$, with $\sum_{i=1}^m a_i=1$, 
   there exists a unique Borel probability measure $\mu_0$ such that 
     \[
       \mu_0(A)=a_1\mu_0(f_1^{-1}(A))+\cdots+a_m\mu_0(f_m^{-1}(A))
     \]
   for every measurable subset $A\subset Y$. In other words, 
   \[
          \mu_0 = \sum_{i=1}^m a_i (f_i)_*(\mu_0),
   \]
   where $ (f_i)_*(\mu_0)$ is the push-forward measure of $\mu_0$ by $f_i$.
%\end{enumerate}
\end{lemma}

\begin{proof}
Define the map 
$F^*(a_1,\ldots,a_m):(\mathcal M(Y), d_{\mathcal M}^*) \to (\mathcal M(Y), d_{\mathcal M}^*)$
by 
\[
   F^*(a_1,\ldots,a_m)(\mu) = \sum_{i=1}^m a_i (f_i)_*(\mu).
\]
If $\phi\in \mathrm{Lip}_1(Y)$, $\phi\circ f_i$ has Lipschitz constant $\le \lambda_{\max}$,
where $\lambda_{\max} =\max\{ \lambda_1,\ldots,\lambda_m\}$.
This implies that  $F^*(a_1,\ldots, a_m)$ is 
$\lambda_{\max}$-contracting,
Since $(\mathcal M(Y), d_{\mathcal M}^*)$ is complete, 
it has a fixed point $\mu_0$ in $\mathcal M(K)$ by the contraction mapping theorem. 
This completes the proof.
\end{proof}

%%%%%%%%%%%%%%%%%  \section{Proof of Main Theorem} %%%%%%%%%%%%%%%%%%%%%%%%%%
\section{Proof of  Theorem \ref{T1}} \label{sec:main}

Let $K$ be the asymptotic self-similar set in  a complete doubling metric space 
 $X$  associated with a $(\{ \lambda_i\}_{i=1}^k, \varphi,\nu)$-asymptotic 
similarity system
 $\{ (\bar V_I, f_I) \}_{I\in\mathcal I^*}$.
For each $I=i_1\cdots i_n\in \mathcal I^*$, we set
\[
   g_I:= f_I\circ \cdots \circ f_{i_1i_2}\circ f_{i_1}:\bar V \to \bar V, \,\,\,
        \bar V_I :=g_I(\bar V)\subset \bar V.
\]
Note that
\begin{equation}
        |V_I|\le \nu^{|I|}|V|.      \label{eq:nu}
\end{equation}

Let $s$ be a unique solution of  $\displaystyle \sum_{i=1}^k \lambda_i^s=1$

\begin{lemma} \label{lem:converge}
Let $\varphi:(0,\infty)\to [0,\infty)$ be a continuous function satisfying the conditions
\eqref{eq:varphi}.  Then
\begin{align*}
     \mathop{\Pi}\limits_{i=0}^{\infty}  (1+\varphi(\nu^i|V|) < \infty,\,\,\,\, 
      \mathop{\Pi}\limits_{i=0}^{\infty}  (1-\varphi(\nu^i|V|)  >0.
\end{align*} 
\end{lemma}
\begin{proof}
By the condition on $\varphi$, we have
\begin{align*}
    \sum_{i=0}^{\infty} \log (1+\varphi(\nu^i|V|)) \le  \sum_{i=0}^{\infty} \varphi(\nu^i|V|) <\infty.
\end{align*}
Similarly we have 
\begin{align*}
    \sum_{i=0}^{\infty} \log (1-\varphi(\nu^i|V|)) \ge -2 \sum_{i=0}^{\infty} \varphi(\nu^i|V|) > -\infty.
\end{align*}
These complete the proof.
\end{proof}

%%%%%%%%%%%%%%%%% $\dim_H K\le s$   %%%%%%%%%%%%%%%%%%%%%%%%%%

Let $I=i_1\cdots i_{m-1} i_m\in \mathcal I^*$.
We use the notation 
\[
        I_{-} = i_1\cdots i_{m-1},  
\]
and write naturally like $I=I_{-}i_m$ as before.

\begin{lemma}
$\dim_HK\le s$
\end{lemma}

\begin{proof}
By the construction, we have $|V_{i_1 \cdots i_n}|\le|V_{i_1 \cdots i_{n-1}}|\nu$. 
For any $\epsilon>0$ take a sufficiently large $n$ 
such that $\mathcal {U}_n:=\{\ V_{I} \ |\ I\in \mathcal I^n\}$
is an $\epsilon$-cover of $K$.
From the definition of $(\lambda_{i_n}, \varphi, \nu)$-almost similarity map 
$f_I:V_{I'}\to V_{I}$, $I=I' i_n$,  we have
\begin{equation*}
|V_{I}|\le\lambda_{i_n}(1+\varphi(|V_{I'}|)|V_{I'}|.
\end{equation*}
It follows from Lemma \ref{lem:converge} that   
\begin{gather*}
\begin{split}
 \mathcal{H}_{\epsilon}^s(K)
& \le  \sum_{I\in\mathcal I^n} |V_{I}|^s   \\
& =    \sum_{I'\in\mathcal I^{n-1}}(\ |V_{I' 1}|^s      + \cdots  + |V_{I' k}|^s\ )       \\
& \le  \sum_{I'\in\mathcal I^{n-1}}(1+\varphi(|V_{I'}|))^s |V_{I'}|^s (\lambda_1^s +\cdots +\lambda_k^s )   \\
& \le  (1+\varphi({\nu}^{n-1}|V|))^s \sum_{I'\in\mathcal I^{n-1}}|V_{I'}|^s   \\
& \le \cdots  < \mathop{\Pi}\limits_{i=0}^{\infty}(1+  \varphi({\nu}^{i}|V|))^s|V| < C|V|,
\end{split}
\end{gather*}
where $C$ is a constant, 
%Hence $\mathcal{H}^s(K)\le C_0$ for some constant $C_0$, 
and therefore $\dim_H K\le s$.\\
\end{proof}

%%%%%%%%% Covering argument %%%%%%%%%%%%%%%%%%%%%%%%%%%%%%%%%%%%%%%%%%%%%%%%%%%%%%%%%%%%%%%%%%%%
\begin{lemma}\label{L3}
Let $X$ be as in Theorem \ref{T1}, and 
let  $\mathcal{V} =\{ V_i  \}$ be a collection of disjoint open sets 
of $X$ such that each $V_i$ contains a closed ball of radius $c_1\rho$ 
and is included in a closed ball of radius $c_2\rho$
for some positive constants $c_1 <c_2$ and  $\rho$. 
Then every closed $\rho$-ball $\bar B(x,\rho)$in $X$ intersects at most $C(\delta)$ elements of 
$\bar{\mathcal  V}=\{ \bar V_i  \}$,  
where $\delta=\frac{c_1}{c_1+4c_2+2}$ and  $C(\delta)$ is a constant given in Lemma \ref{lem:doubling}.
\end{lemma}

\begin{proof}
Take $x_1^i,x_2^i\in X$ satisfying  $\bar B(x_1^i,c_1\rho)\subset V_i\subset \bar B(x_2^i,c_2\rho)$. 
Let $\bar{V}_1,\cdots,\bar{V}_N$ intersect $\bar B(x,\rho)$.\\
Taking any  point $z\in\bar{V_i}\cap \bar B(x,\rho)$,  we have 
\begin{gather*}
\begin{split}
d(x_1^i,x)
\le d(x_1^i,z)+ d(z,x) \le (2c_2+1)\rho.
\end{split}
\end{gather*}
Furthermore, for any $y\in B(x_1^i,c_1\rho)$, we have 
\begin{gather*}
\begin{split}
d(y,x)
\le d(y,x_1^i)+ d(x_1^i,x) < (c_1+2c_2+1)\rho.
\end{split}
\end{gather*}
Thus  we get 
\begin{equation*}
\displaystyle\bigcup _{i=1}^NB(x_1^i,c_1\rho)\subset B\big(x,(c_1+2c_2+1)\rho\big).
\end{equation*}
Since $B(x_1^i,c_1\rho)$ are mutually disjoint,  from Lemma \ref{lem:doubling} we obtain 
the conclusion of the lemma. This completes the proof.
\end{proof}

%%%%%%%  $\dim_H K\ge s$. %%%%%%%%%%%%%%%%%%%%%%%%%%%%%%%%%%%%%%%%%%%%%%%%%%%%%%%%%%%%%%

The rest of this section is mainly devoted to prove the following.

\begin{lemma} \label{lem:from-below}
$\dim_H K\ge s$.
\end{lemma}

We set  
\[
 \bar V^{n}:=\bigcup_{I\in \mathcal I^n} \bar V_I.
\]
Note that 
\[
           K=\bigcap_{n=1}^{\infty} \bar V^{n}.
\]
For a large $n_0$, fix an abitrary $I_0=i_1\cdots i_{n_0}\in\mathcal I^{n_0}$, 
and consider 
\[
             \bar V_{I_0} := g_{I_0}(\bar V)=f_{I_0}\circ \cdots f_{i_1i_2}\circ f_{i_1}(\bar V),\,\,\,
                K_{I_0}:=K\cap V_{I_0}.
\]
It suffices to prove that 
$\dim_H K_{I_0}\ge s$. 
Therefore we start with
\[
                        W:= V_{I_0},
\]
instead of $V$.

For every $1\le i\le k$, put 
\[
     h_i:=f_{I_0i}:\bar W\to \bar W_{i},
\]
where
\[
                \bar W_{i} := h_i(\bar W)\subset \bar W.
\]
Recall from the definition
\[
       \left| \frac{d(h_i(x), h_i(y))}{d(x,y)} - \lambda_i\right| < o(n_0),
\]
for every $x\neq y\in\bar W$, 
where 
\begin{equation}
                      o(n_0)=\lambda_{\max}\varphi(\nu^{n_0}|V|), \label{eq:o}
\end{equation}
and therefore
$\lim_{n_0\to\infty} o(n_0)=0$. 
For $J=j_1\cdot \cdot j_m\in \mathcal I^*$ and every $1\le \ell\le m$, 
 we use the notation 
\[
 h_{j_1\cdot \cdot j_{\ell}}:= f_{Ij_1\cdot \cdot j_{\ell}}
:\bar W_{j_1\cdot \cdot j_{\ell-1}} \to \bar W_{j_1\cdot \cdot j_{\ell}},
\]
as before, and define 
$g_J:\bar W \to \bar W_J$ by 
\[
                           g_J:= h_J\circ \cdots \circ h_{j_1j_2}\circ h_{j_1}.
\]

\begin{lemma} \label{lem:g-dil}
For every $x\neq y\in\bar W$, we have 
\[ 
      \left| \frac{d(g_J(x), g_J(y))}{d(x,y)} - \lambda_J\right| < o(n_0)\lambda_J,
\]
where $\lambda_J=\lambda_{j_1}\cdots\lambda_{j_m}$.
\end{lemma}

\begin{proof}
Put $J_{\ell}:=j_1\cdot \cdot j_{\ell}$ for each $1\le \ell\le m$.
From Lemma \ref{lem:converge}, we obtain
\begin{align*}
  \frac{d(g_J(x), g_J(y))}{d(x,y)}  
        &= \frac{d(g_{J_m}(x), g_{J_m}(y))}{d(g_{J_{m-1}}(x), g_{J_{m-1}}(y))}\cdots
                         \frac{d(g_{J_2}(x), g_{J_2}(y))}{d(g_{J_{1}}(x), g_{J_{1}}(y))} \frac{d(g_{J_1}(x), g_{J_1}(y))}{d(x, y)} \\
        & \le \lambda_J\,  \mathop{\Pi}\limits_{\ell=0}^{\infty} (1+\varphi(\nu^{n_0+\ell}|V|))  \\
        & =   \lambda_J(1+o(n_0)).       
\end{align*}
An estimate from below is  similar, and hence omitted.
\end{proof}

For a small $\epsilon>0$ compared with $|W|$, 
let $\{ U_i\}$ be any $\epsilon$-covering of 
\[
                \tilde K:=K_{I_0}.
\]
Replacing $U_i$ by balls $B_i$ of radius $2|U_i|$, we have
a covering $\{ B_i\}$ of $\tilde K$. Thus 
\[
  \sum |U_i|^s \ge 2^{-s}\sum |B_i|.
\]
Fix $B_i$ and take $c_1>0$ and $c_2>0$ such that
$W$ contains a ball of radius $c_1|W|$
and is contained in a ball of radius $c_2|W|$.

\begin{definition} \upshape
We denote by $\mathcal I^{\infty}$ the set of all 
infinite sequences $J=j_1j_2\cdots $ with $j_{\ell}\in\mathcal I$ 
for all $\ell\ge 1$.
We call a finite subset $\mathcal S$ of  $\mathcal I^*$
a {\it  simple family} if
%%\begin{enumerate}
% \item  
for each $J=j_1j_2\cdots \in \mathcal I^{\infty}$,
   there is a unique $m$ such that $J_m=j_1j_2\cdots j_m\in\mathcal S$.
%\item the set of all $J_m$ as above when $J$ runs over $\mathcal I^{\infty}$ coincides with $\mathcal S.$
%\end{enumerate} 
% the following conditions are satisfied:
%\begin{enumerate}
% \item  $\displaystyle\bigcup_{I\in\mathcal S} \bar W_I \supset\tilde K;$
% \item If $I=i_1\cdots i_{m-1}i_m\in \mathcal S$, then
%neither $I_0=i_1\cdots i_{m-1}$ nor  $I_1=i_1\cdots i_{m-1}i_mi$
%belong to $\mathcal S$ for all $1\le i\le k$.
%\end{enumerate} 
\end{definition}

For instance, $\mathcal I^m$ is  a simple family for every $m\ge 1$.

\begin{lemma} \label{lem:simple}
For every simple family $\mathcal S$, we have
\[
       \sum_{I\in\mathcal S} \lambda_I^s =1.
\]
\end{lemma}

\begin{proof}
Let $m:=\max_{I\in\mathcal S} |I|$. We prove the lemma by the reverse induction on $m$.
Take $I\in\mathcal S$ with $|I| = m$, and let $I=i_1\cdots i_m$.
Recall $I_{-} = i_1\cdots i_{m-1}$ and note 
that 
$I_-j\in\mathcal S$ for all $j\in \mathcal I$.
It follows that 
\[
       \sum_{j=1}^k \lambda_{I_-j}^s = \lambda^s_{I_-}.
\]
Set 
\[
    \mathcal S_m :=\mathcal S\cap \mathcal I^m,\,\,\, 
    \mathcal S':=(\mathcal S\setminus \mathcal S_m)\cup\{ I_-\,|\, I\in \mathcal S_m\}.
\]
Since $\mathcal S'$is a simple family, it follows from the inductive hypothesis that 
\[
 \sum_{I\in\mathcal S} \lambda_I^s = \sum_{I\in\mathcal S'} \lambda_I^s =1
\]
\end{proof}

\begin{assertion} \label{ass:simple}
For each $i$, there is a simple family $\mathcal S = \mathcal S_i$
consisting of $J$ satisfying that $\bar W_J$ is contained in 
a ball of radius $c_2|B_i|$ and  contains a ball of radius 
$\tilde\lambda_{min} c_1c_2|B_i|$  for some uniform 
constant $0<\tilde\lambda_{min}\le\lambda_{\min}$.
\end{assertion}

\begin{proof}
For each $J=j_1j_2\cdots \in \mathcal I^{\infty}$,
there is a unique $m$ such that 
\begin{align}
   |W_{j_1\cdots j_{m-1}}|> c_2|B_i|, \,\,\, |W_{j_1\cdots j_{m}}|\le  c_2|B_i|. \label{eq:W}
\end{align}
Set $J_m:=j_1\cdots j_{m}$. Obviously, $W_{J_m}$ is contained in  
a ball of rdius $c_2|B_i|$.
Since $W$ contains a ball of radius $c_1|W|$ and since $W_{J_m}$ is open,
$W_{J_m}$ contains a ball of radius $(1-o(n_0))\lambda_Jc_1|W|$.
From the choice of $J_m$, 
\[
    (1-o(n_0))\lambda_Jc_1|W|\ge (1-o(n_0))^2\lambda_{j_m}c_1c_2|B_i|.
\]
Let $\mathcal S$ be the set of all $J_m\in \mathcal I^*$ when $J$ runs over $\mathcal I^{\infty}$. \eqref{eq:W} implies that $\nu^{m-1}\ge c_2|B_i|/|W|$, and therefore 
 $\mathcal S$ is finite.
This completes the proof.
\end{proof}

Applying Lemma \ref{lem:similar-measure} to the contracting maps 
$g_I:\bar W \to \bar W$, $I\in\mathcal S$, we have 

\begin{assertion} \label{ass:meas}
Let $\mathcal S=\mathcal S_i$ be as in Assertion \ref{ass:simple}. Then 
there is a unique Borel probability measure $\mu=\mu_{\mathcal S}$
in $\mathcal M(\bar W)$ such that
\[
    \mu = \sum_{I\in\mathcal S} \lambda_I^s (g_I)_*(\mu),
\]
where $\lambda_I^s=(\lambda_I)^s$.
\end{assertion}

%\begin{proof}
%Define $F:(\mathcal M(\bar W), d_{\mathcal M}^*) \to (\mathcal M(\bar W), d_{\mathcal M}^*)$ by
%\[
%    F(\sigma) = \sum  \lambda_I^s (g_I)_*(\sigma).
%\]
%It is straightforward to see that  $F$ is contracting. 
%Then the conclusion follows from the contraction mapping theorem.
%\end{proof}
%
Since $\bar W\supset \tilde K$, it follows from Lemma \ref{lem:g-dil} and the property of $\mathcal S$ that
for any $J\in \mathcal S$,
\begin{equation}
   2^s c_2^s |B_i|^s \ge |W_J|^s \ge |\tilde K_J|^s \ge (1-o(n_0))\lambda_J^s |\tilde K|^s. \label{eq:1}
\end{equation}
By Lemma \ref{L3},  the number of $\bar W_J$  with $J\in \mathcal S$ meeting  
$B_i$ is uniformly bounded by some constant $C=C(\delta)$, where 
$\delta=\delta(c_1,c_2,\tilde\lambda_{\min})$.
Let $\mu$ be the measure constructed in Assertion  \ref{ass:meas}.
Then we have 
\begin{equation}
\begin{aligned}
   \mu(B_i)& =  \sum_{I\in\mathcal S} \lambda_I^s (g_I)_*(\mu)(B_i) 
              =  \sum_{I\in\mathcal S} \lambda_I^s  (g_I)_*(\mu)(B_i\cap \bar W_I) \\
             & \le C(\delta)\max_{I\in\mathcal S, \bar W_I\cap B_i\neq \phi} \lambda_I^s. \label{eq:2}
\end{aligned}
\end{equation}
It follows from \eqref{eq:1} and \eqref{eq:2} that
\begin{align}
    c_2^s |B_i|^s \ge (1-o(n_0))C(\delta)^{-1}|K|^s\mu(B_i). \label{eq:3}
\end{align}
Since 
\begin{align}
           \sum_{|J|=m} \lambda^s_J =1, \label{eq:m}
\end{align}
for each $m\ge 1$, applying Lemma \ref{lem:similar-measure} to the contracting maps 
$g_J:\bar W\to \bar W$, $J\in\mathcal I^m$, 
we have a unique measure $\mu_m\in\mathcal M(\bar W)$ such that
\[
     \mu_m = \sum_{|J|=m} \lambda_J^s (g_J)_*(\mu_m).
\]

\begin{assertion}
For $m>\max_{I\in \mathcal S} |I|$, we have $\mu= \mu_m$.
\end{assertion}

\begin{proof}
For  each $J\in \mathcal I^m$, there are unique $I\in \mathcal S$ and
$J_{\alpha}\in\mathcal I^*$ such that $J=IJ_{\alpha}$. Let $A_I$ be the set of all the  
indices $\alpha$ with  $J=IJ_{\alpha}$ for some $J\in\mathcal I^m$
 We can write as 
\[
      \mu_m = \sum_{I\in\mathcal S, \alpha\in A_I} \lambda_{IJ_{\alpha}}^s
 (g_{IJ_{\alpha}})_*(\mu_m).
\]
By iterating $\ell$-times, we have
\begin{align*}
   \mu_m= & \sum_{J_1,\ldots,J_{\ell}\in\mathcal I^m} \lambda_{J_1}^s\cdots
   \lambda_{J_{\ell}}^s (g_{J_{1}}\circ\cdots\circ g_{\ell})_* (\mu_m) \\
          = &\sum_{ I_i\in\mathcal S, \alpha_i\in A_{I_i}}
%  &=\sum_{ I_1,\ldots, I_{\ell},\alpha_1, \ldots, \alpha_{\ell}} 
   \lambda_{I_1J_{\alpha_1}}^s\cdots
   \lambda_{I_{\ell}J_{\alpha_{\ell}}}^s 
    (g_{J_{1}}\circ\cdots\circ g_{J_{\ell}})_{*} (\mu_m).
\end{align*}
Since $A_I = \mathcal I^{m-|I|}$, similarly to \eqref{eq:m} we see  
\begin{align}
         \sum_{\alpha\in A_I} \lambda_{J_{\alpha}}^s = 1. \label{eq:A}
\end{align}
It follows that 
\begin{align*}
   \mu =\sum_{I\in \mathcal S} \lambda_{I}^s (g_I)_*(\mu) 
     =\sum_{ I\in\mathcal S, \alpha\in A_{I}}  \lambda_{IJ_{\alpha}}^s (g_{I})_{*} (\mu).
\end{align*}
By iterating $\ell$-times, we obtain
\begin{align*}
   \mu =\sum_{ I_i\in\mathcal S, \alpha_i\in A_{I_i}} 
  \lambda_{I_1J_{\alpha_1}}^s\cdots
   \lambda_{I_{\ell}J_{\alpha_{\ell}}}^s 
    (g_{I_{1}}\circ\cdots\circ g_{I_{\ell}})_{*} (\mu).
\end{align*}
It follows that
\begin{align*}
d_{\mathcal M}^*(\mu, \mu_m) \le
 & \sum_{ I_i\in\mathcal S, \alpha_i\in A_{I_i}}  
 \lambda_{I_1J_{\alpha_1}}^s\cdots \lambda_{I_{\ell}J_{\alpha_{\ell}}}^s  \\
  & \sup_{L(\phi)\le 1} \left|\int \phi\circ g_{I_{\ell}}\circ\cdots\circ
    g_{I_1}\,d\mu - \int \phi\circ g_{J_{\ell}}\circ\cdots\circ g_{J_1}\,d\mu_m\right|
\end{align*}
Here,
\begin{align*}  
 | \int \phi\circ g_{I_{\ell}} & \circ\cdots\circ
  g_{I_1}\,d\mu -  \int \phi\circ g_{J_{\ell}}\circ\cdots\circ
  g_{J_1}\,d\mu_m | \\
&\le
\left|\int\phi\circ g_{I_{\ell}}\circ\cdots\circ g_{I_1}\,d\mu - 
   \int \phi\circ g_{I_{\ell}}\circ\cdots\circ g_{I_1}\,d\mu_m\right| \\
&+
\left|\int\phi\circ g_{I_{\ell}}\circ\cdots\circ g_{I_1}\,d\mu_m - 
   \int \phi\circ g_{J_{\ell}}\circ\cdots\circ g_{J_1}\,d\mu_m\right|.
\end{align*}
For a constant $\tilde\lambda$ with $\lambda_{\max}<\tilde\lambda<1$,
choose a large $n_0$ such that 
$(1+o(n_0))\lambda_{\max}<\tilde\lambda<1$. 
%for some uniform constant $\tilde\lambda<1$.
%
Then the Lipschitz constant of $g_{I_{\ell}}\circ\cdots\circ g_{I_1}$ satisfies 
\[
 L(g_{I_{\ell}}\circ\cdots\circ g_{I_1} )\le 
              (1+o(n_0))^{\ell}\lambda_{I_\ell}\cdots\lambda_{I_1}
                          <\tilde\lambda_{I_1\cdots I_{\ell}},
\]
where we put $\tilde\lambda_{I_1\cdots I_{\ell}} :=(\tilde\lambda)^{|I_1|+\cdots+|I_{\ell}|}$.
Therefore we obtain
\begin{align*}
 |\int\phi\circ g_{I_{\ell}}&\circ\cdots\circ g_{I_1}\,d\mu - 
   \int \phi\circ g_{I_{\ell}}\circ\cdots\circ g_{I_1}\,d\mu_m|\\
     &\le \tilde\lambda_{I_1\cdots I_{\ell}}  d_{\mathcal M}^*(\mu, \mu_m).
\end{align*}
On the other hand, from the inclusion
\[
 g_{I_{\ell}}\circ\cdots\circ g_{I_1}(\bar W)
   \supset g_{J_{\ell}}\circ\cdots\circ g_{J_1}(\bar W),
\]
we have
\begin{align*}
\sup_{x\in\bar W} |\phi\circ g_{I_{\ell}}&\circ\cdots\circ g_{I_1}(x) -  
   \phi\circ g_{J_{\ell}}\circ\cdots\circ g_{J_1}(x)|  \\
   &\le |g_{I_{\ell}}\circ\cdots\circ g_{I_1}(\bar W)|  \\
   &\le  (1+o(n_0))^{\ell}\lambda_{I_\ell}\cdots\lambda_{I_1}
                          <\tilde\lambda_{I_1\cdots I_{\ell}}
\end{align*}
Thus letting $n=\min_{I\in\mathcal S} |I|$ together with \eqref{eq:A},  we have 
\begin{align*}
 d_{\mathcal M}^*(\mu, \mu_m) &\le
   \sum_{ I_1,\ldots, I_{\ell},\alpha_1, \ldots, \alpha_{\ell}} 
    \lambda_{I_1J_{\alpha_1}}^s\cdots
         \lambda_{I_{\ell}J_{\alpha_{\ell}}}^s
         \tilde\lambda_{I_1\cdots I_{\ell}} ( d_{\mathcal M}^*(\mu, \mu_m)+1)  \\
   &\le \tilde\lambda^{n\ell}
       \sum_{ I_1,\ldots, I_{\ell},\alpha_1, \ldots, \alpha_{\ell}} 
    \lambda_{I_1J_{\alpha_1}}^s\cdots
         \lambda_{I_{\ell}J_{\alpha_{\ell}}}^s ( d_{\mathcal M}^*(\mu, \mu_m)+1)  \\
   &= \tilde\lambda^{n\ell}
     \sum_{ I_1,\ldots, I_{\ell}\in\mathcal S} 
    \lambda_{I_1}^s\cdots \lambda_{I_{\ell}}^s  ( d_{\mathcal M}^*(\mu, \mu_m)+1)\\
   & = \tilde\lambda^{n\ell} ( d_{\mathcal M}^*(\mu, \mu_m)+1),
\end{align*}
which yields
\[
     d_{\mathcal M}^*(\mu, \mu_m) < \frac{1}{1- \tilde\lambda^{n\ell}}  \tilde\lambda^{n\ell}.
\]
Letting $\ell\to\infty$, we conclude that
$\mu= \mu_m$.
\end{proof}
\begin{proof}[Proof of Lemma \ref{lem:from-below}]
From the last assertion, we have
\[
 \supp(\mu) \subset \bigcap_{m=1}^{\infty}\left(\bigcup_{|J|=m} g_J(\bar W)\right)=\tilde K.
\]
It follows from \eqref{eq:3} that
\begin{align*}
   \sum 2^{-s} |B_i|^s & \ge (1-o(n_0))4^{-s}c_2^{-s}C(\delta)^{-1}|\tilde K|\sum \mu(B_i) \\
         & \ge (1-o(n_0))4^{-s}c_2^{-s}C(\delta)^{-1}|\tilde K|.
\end{align*}
This shows that $\dim_H \tilde K \ge s$. We have completed the proof of lemma \ref{lem:from-below}.
\end{proof}
\bigskip

Finally we show that

\begin{lemma} 
$\overline{\dim}_B K \le s$. \label{lem:from-above}
\end{lemma}

\begin{proof}
For every $\epsilon>0$ and $J_{\infty}=j_1j_2\cdots\in\mathcal I^{\infty}$, take a minimal
$m$ satisfying $|W_{J}|\le\epsilon$ for $J:=J_m=j_1\cdots j_m$.
Note that 
\begin{align}
 |W_{J}| \ge \lambda_{\min}/2|W_{J_{-}}| \ge  \epsilon\lambda_{\min}/2.  \label{eq:below}
\end{align}
Thus we have a simple family $\mathcal S=\{ J\,|\,J_{\infty}\in\mathcal I^{\infty}\,\}$.
By Lemma \ref{lem:simple}, we have  
\begin{align}
     \sum_{J\in\mathcal S}\,\lambda_{J}^s =1. \label{eq:sumJ'}
\end{align}
By Lemma \ref{lem:g-dil}, we have 
\begin{align}
    \left| \frac{|W_{J}|}{|W|}  -\lambda_{J} \right| < \lambda_{J}o(n_0). \label{eq:W}
\end{align}
It follows from \eqref{eq:below} and  \eqref{eq:W} that
\[
     ( \epsilon\lambda_{\min}/2)^s \le 2^s\lambda_{J}^s|W|^s.
\]
Using \eqref{eq:sumJ'}, we obtain
\[
      \sum_{J\in\mathcal S} (\epsilon\lambda_{\min}/2)^s\le 2^s|W|^s.
\]
Since $\{ W_J\,|\, J\in \mathcal S\}$ is disjoint, we conclude that 
\[
         N_{\epsilon}(\tilde K)\le  2^s|W|^s (\epsilon\lambda_{\min}/2)^{-s}.
\]
This shows that $\overline{\dim}_B\,\tilde K\le s$, and the conclusion of the lemma follows.
\end{proof}

It follows from Lemmas \ref{lem:from-below}, \ref{lem:from-above}  and \eqref{eq:dims} that 
$\dim_H\,K =\dim_B\,K=s$.
This completes the proof of Theorem \ref{T1}.
\par
\bigskip
Finally we point out that our notion of  asymptotic similarity 
system  provides a controlled Moran construction defined in
Rajala and Vilppolainen \cite{RV}:

\begin{lemma} \label{lem:remarkRV}
Let $\{ (\bar V_I, f_I) \}_{I\in\mathcal I^*}$ be a $(\{ \lambda_i\}_{i=1}^k, \varphi,\nu)$-asymptotic 
similarity system. Then  $\{ \bar V_I \}_{I\in\mathcal I^*}$ is a controlled Moran construction 
defined in  Rajala and Vilppolainen (\cite{RV}).
Namely, there exists a constant $D\ge 1$ such that for every $I,J \in \mathcal I^*$
\begin{enumerate}
 \item $\displaystyle{\bar V_I\subset \bar V_{I^{-}}}$;
 \item there exists a positive integer $n$ such that
  \[
                      \max_{I\in \mathcal I^n}\, |\bar V_I| <D^{-1};
 \]
 \item   $\displaystyle{D^{-1} \le \frac{|\bar V_{IJ}|} {|\bar V_{I}||\bar V_{J}|} \le D}$.
\end{enumerate}
\end{lemma}

\begin{proof}
(1) is clear. In view of \eqref{eq:nu}, (2) is obvious.
To show (3), we go back to the situation of Lemma \ref{lem:g-dil}.
Let $o(n_0)$ be as in \eqref{eq:o}. 
For a large $n_0$, fix an abitrary $I_0=i_1\cdots i_{n_0}\in\mathcal I^{n_0}$,
and consider $W=V_{I_0}$.
If we take $n_0$ with $o(n_0)<1/2$, we have from Lemma \ref{lem:g-dil},
\[
        \frac{1}{2}\lambda_I |\bar W| <|\bar W_I| <2\lambda_I \bar W|,\quad
             \frac{1}{2}\lambda_J |\bar W| <|\bar W_J| <2\lambda_J |\bar W|,
\]
which imply 
\[
    \frac{1}{4|\bar W|} |\bar W_I||\bar W_J| < |\bar W_{IJ}| < \frac{4}{|\bar W|} |\bar W_I||\bar W_J|.
\]
Now (3) is immediate, since we have only finitely many choices for $I_0$.
\end{proof}

%%%%%%%%%%%%%%%%%%%%%%%%%%%%%%%%%%%%%%%%%%%%%%%%%%%%%%%%%%%%%%%%%%%%%%%%%%%%%%%%%%%%%%%%%%%%%%%%%%
\section{Sierpinski gaskets on surfaces} \label{sec:Sierp}

Let $D$ be a domain in a complete surface $M$. We assume that $D$ is convex in the 
sense that for every two points of $D$ there exits a unique minimal geodesic joining them
and it is contained in $D$.
%for every $p\in D$, the distance function $d_p(\cdot)=d(p, \cdot)$ from $p$
%is convex in $D$.   
For simplicity, we assume that
the absolute value of the Gaussian curvature of $M$ is at most $1$ on $D$.
Let $\Delta$ be a domain in $D$ bounded by a geodesic triangle $(\gamma_1,\gamma_2,\gamma_3)$.
We call $\Delta$ a {\it geodesic triangle region}. The set of lengths $\{ L(\gamma_i)\}_{i=1}^3$ 
is called the side-length of $\Delta$.

\begin{definition} \upshape
We say that $\Delta$ is 
{\it $\delta$-non-degenerate}
if each angle $\tilde \alpha$ of a comparison triangle $\tilde \Delta$ of $\Delta$ in $\mathbb R^2$ satisfies
$\delta < \tilde\alpha < \pi - \delta$, where a comparison triangle means that $\tilde \Delta$ has 
the same side-length as $\Delta$. 
\end{definition}

In this section, we let $\mathcal I=\{  1,2,3\}$.
 Let  $\{ \Delta_I\}_{I\in\mathcal I^*}$ be the system of geodesic triangles obtained by dividing $\Delta$
into smaller triangles $\Delta_I$ consecutively, as stated in Introduction.

\begin{definition} \upshape
We say that the system $\{ \Delta_I\}_{I\in\mathcal I^*}$ is 
{\it non-degenerate} if there is a $\delta>0$ such that
$\Delta_I$ is $\delta$-non-degenerate for every $I\in\mathcal I^*$.
In this case, we also say that $\Delta$ is {\it asymptotically 
non-degenerate}.
\end{definition}

\begin{example} \label{ex:sphere}  \upshape
Let $\mathbb S^2$ denote the unit sphere around the origin in $\mathbb R^3$, 
and let $\Delta$ be a geodesic triangle domain on $\mathbb S^2$
of perimeter less than $2\pi$. Joining the vertexes $p_1,p_2,p_3$ of $\Delta$
by shortest segments in $\mathbb R^3$, we have a geodesic triangle region
$\hat \Delta$ on the plane through $p_1,p_2,p_3$. 
By the projection along the rays from the origin of $\mathbb R^3$, 
we have a canonical map 
\[
         \pi:\Delta\to\hat\Delta,
\]
which is a bi-Lipschitz homeomorphism. 
From a system of geodesic triangles 
$\{ \Delta_I\}_{I\in\mathcal I^*}$ of $\Delta$, setting $\hat\Delta_I:=\pi(\Delta_I)$,
we have the  system of geodesic triangles  $\{ \hat \Delta_I\}_{I\in\mathcal I^*}$ of 
$\tilde\Delta$. Note that each $ \hat\Delta_I$ is $2^{-|I|}$-similar to $\hat \Delta$
in the usual sense. 
Since $\Delta_I$ is bi-Lipschitz homeomorphic to 
$\hat \Delta_I$, 
\[
         {\rm Area}(\Delta_I)\ge L^{-2}{\rm Area}(\hat\Delta_I),
\]
where $L$ is the bi-Lipschitz constant of $\pi$.
It follows that $\Delta$ is asymptotically non-degenerate. 
Now we have the formula \eqref{eq:dimHB} for the 
Sierpinsli gasket $K_{\Delta}$ associated with $\Delta$ by  two reasons.
One is by Theorem \ref{thm:Sierp}
and the other one is due to the 
well-known formula for  $K_{\hat \Delta}$.
\end{example}

Example \ref{ex:sphere} is the special case. For a geodesic triangle region on 
a general complete surface, 
it  seems impossible to reduce the problem to a triangle region in $\mathbb R^2$.

The main purpose of this section is to prove the following result.

\begin{theorem} \label{thm:Sierp-small}
For every $\delta>0$ there exists an $r>0$ such that
\begin{enumerate}
 \item every geodesic triangle region $\Delta$ on $D$ with $|\Delta|\le r$ is asymptotically non-degenerate;
 \item the Hausdorff and box dimensions of the Sierpinski gasket $K_{\Delta}$ associated with $\Delta$ 
         are  given by \eqref{eq:dimHB}.
\end{enumerate}
\end{theorem}

If $\Delta$ be asymptotically non-degenerate as in Theorem \ref{thm:Sierp},
we can apply Theorem \ref{thm:Sierp-small} to $\Delta_I$ for  each $I\in\mathcal I^*$
with large enough $|I|$.
Therefore Theorem \ref{thm:Sierp-small} yields Theorem \ref{thm:Sierp}.

 The following lemma is a consequence of law of cosine, and hence is omitted.

\begin{lemma}\label{lem:nondeg}
For any $\delta>0$ there exists an $\epsilon>0$ such that 
if a geodesic triangle $\Delta$ of side length $(a_1, a_2, a_3)$ is $\delta$-non-degenerate,
and if the  side length $(a_1', a_2', a_3')$ of a geodesic triangle $\Delta'$ satisfies
\begin{equation}
        (1-\epsilon)  \frac{a_j}{a_i} <  \frac{a'_j}{a'_i}  <   (1+\epsilon)  \frac{a_j}{a_i}, \label{eq:edge-assum}
\end{equation}
for any $i\neq j$, then $\Delta'$ is $\delta/2$-non-degenerate.
\end{lemma}

\begin{proof}
We may assume that  $\Delta$ and $\Delta'$ are triangles in $\mathbb
 R^2$. Set $(a,b,c):=(a_1, a_2, a_3)$ and $(a',b',c'):=(a_1', a_2', a_3')$ for simplicity.
Rescaling $\Delta'$, we may assume that $c=c'$.
It suffices to show that  if $\Delta'$ has  side-length
 $(a',b',c')=(a',b,c)$ satisfying \eqref{eq:edge-assum}, then  
the angles $\alpha$, $\beta$ (resp. $\alpha'$, $\beta'$) 
opposite to the edges of length $a$ and $b$ in $\Delta$ (resp  $a'$ and $b$ in $\Delta'$)
satisfy  that $|\alpha'-\alpha|<\delta/4$ and $|\beta'-\beta|<\delta/4$
for a suitable $\epsilon =\epsilon(\delta)>0$.

\begin{sublemma} \label{slem:edge-quot}
If a geodesic triangle $\Delta$ of side lengths $(a_1, a_2, a_3)$ is $\delta$-non-degenerate,
then there exists a constant $C(\delta)$ such that 
\[
         C(\delta)^{-1} < \frac{a_j}{a_i} < C(\delta),
\]
for every $1\le i,j \le 3$.
\end{sublemma}

\begin{proof}
This is an immediate consequence of the law of sines. One can take $C(\delta)=1/\sin\delta$.
\end{proof}

By trigonometry, we have
\begin{equation*}
  \sin^2\alpha/2 = (a+c)(a+b)/bc,\,\, \sin^2\alpha'/2 = (a'+c)(a'+b)/bc.
\end{equation*}
It follows from the assumption and Sublemma \ref{slem:edge-quot} with $|a'-a|<\epsilon a$ that
\begin{equation}
    |\sin^2\alpha'/2 -  \sin^2\alpha/2|  \le a(a+a'b+c)\epsilon/bc \le 5C(\delta)^2\epsilon.
\end{equation}
Since $\sin \alpha'/2 +\sin \alpha/2 > \sin (\delta/2)$, we obtain
\[
   |\sin\alpha'/2 -  \sin\alpha/2|\le  5C(\delta)^2\epsilon/\sin (\delta/2).
\]
From $\alpha<\pi-2\delta$, we have $\cos\frac{\alpha'+\alpha}{4} > \sin (\delta/4)$.
It follows that 
\begin{equation}
 |\alpha'-\alpha| \le 8
  \left| \sin\frac{\alpha'-\alpha}{4} \right|  < 5C(\delta)^2\epsilon/\sin^2(\delta/4).\label{eq:alpha}
\end{equation}
Similarly we have
\begin{align*}
    |\sin^2\beta'/2 -  \sin^2\beta/2|  &= |a-a'|b(b+c)/aa'c \le b(b+c)\epsilon/ca' \\
                         & \le \frac{\epsilon}{1-\epsilon} \frac{b(b+c)}{a} \le  \frac{\epsilon}{1-\epsilon} 2 C(\delta)^2,
\end{align*}
which implies 
\begin{equation}
     |\beta' - \beta| <\frac{8\epsilon}{1-\epsilon} \left( \frac{C(\delta)}{\sin (\delta/2)}\right)^2.
                                           \label{eq:beta}
\end{equation}
Thus from \eqref{eq:alpha}, \eqref{eq:beta},  we obtain
$|\alpha'-\alpha|<\delta/4$ and $|\beta'-\beta|<\delta/4$ for a suitable $\epsilon\le\epsilon(\delta)$.
This completes the proof.
\end{proof}

Let $\Delta$ be  a geodesic triangle region on $D$ bounded by a geodesic triangle
$(\gamma_1,\gamma_2,\gamma_3)$ with vertices $p_1,p_2,p_3$.
By the convexity of $D$, we have  
\[
        |\Delta|=\max _{1\le i\le 3} a_i,
\]
where we put $a_i:=L(\gamma_i)$.
Fix a vertex $p_1$ and let $\gamma_i$ be parametrized on $[0,1]$  in such a way that  
$\gamma_2(0)=\gamma_3(0)=p_1$. 
Let $\varphi:[0,1]\times [0,1]\to \Delta$ be a parametrization of $\Delta$ such 
that $t\to \varphi(t,s)$, $0\le t\le 1$, is the geodesic, denoted by $\sigma_s$,
from $\gamma_2(s)$ to $\gamma_3(s)$ for each $s\in [0,1]$.
Namely $\varphi(t,s)=\sigma_s(t)$.  We set 
\[
                    a_1(s) := L(\sigma_s).
\]
Now define the map $f_1:\Delta\to \Delta$ by 
\[
     f_1(\varphi(t,s))=\varphi(t,s/2).
\]
Note that the image $\Delta_1$ of $f_1$ is the geodesic triangle region  bounded by 
$(\gamma_2|_{[0,1/2]}, \gamma_3|_{[0,1/2]}, \sigma_{1/2})$ and that  $\Delta_1$ has side-length
$( a_1(1/2), a_2/2, a_3/2)$. We put
\[
          r:= |\Delta|.
\]

\begin{lemma}\label{lem:Rauch} 
For any $s\in (0,1)$,  we have
\[
       1- r^2 <\frac{a_1(s)}{sa_1} < 1+r^2.
\]
In particular, $|\Delta_1|\le \frac{1}{2}(1+r^2)|\Delta|$.
\end{lemma}

\begin{proof}
Let $\tilde\gamma_i(s):=\exp_{p_1}^{-1}(\gamma_i(s))$, $i=2,3$.
The Rauch comparison theorem (see \cite{CHE}) implies 
\begin{align}
    \frac{\sin r}{r} &<\frac{a_1}{d(\tilde\gamma_2(1),
 \tilde\gamma_3(1))} <\frac{\sinh r}{r}            \\
   \frac{\sin r}{r} &<\frac{a_1(s))}{d(\tilde\gamma_2(s),
 \tilde\gamma_3(s))} <\frac{\sinh r}{r}.  
\end{align}
Since $d(\tilde\gamma_2(s), \tilde\gamma_3(s))=sd(\tilde\gamma_2(1),
 \tilde\gamma_3(1)$,  the conclusion follows.
\end{proof}

Let us denote by $(a_{1,1}, a_{1,2}, a_{1,3})$ the side length $(a_1(1/2), a_2/2, a_3/2)$
of  $\Delta_1$. Lemma \ref{lem:Rauch} implies that
\begin{equation}
               (1-r^2)\frac{a_i}{a_j} <\frac{a_{1,i}}{a_{1,j}} <(1+r^2)\frac{a_i}{a_j}, \label{eq:quotient1}
\end{equation}
for every $1\le i, j\le 3$. 

In a similar  way, we construct a map $f_{i_1}:\Delta\to \Delta_{i_1}\subset \Delta$ for 
each $1\le i_1\le 3$.
Repeating this procedure for each $\Delta_i$ inductively, for each multi-index $I=i_1\cdots i_{n-1}i_n$,
we have a geodesic triangle region $\Delta_I$ 
and a map $f_I:\Delta_{I'} \to \Delta_I$, where $I'=i_1\cdots i_{n-1}$.
The side-length  $(a_{I,1}, a_{I,2}, a_{I,3})$ of $\Delta_I$ is also suitably defined inductively.
Take $r<1$ and set   
\[
       \nu := \frac{1}{2}(1+r^2) <1.  
\]

\begin{lemma}   \label{lem:quotient}
There exists an  $L(r)>1$ such that for every $I$ and  $1\le i,j\le 3$
\begin{equation*}
       L(r)^{-1} \frac{a_{i}}{a_{j}} < \frac{a_{I,i}}{a_{I,j}} < L(r)\frac{a_{i}}{a_{j}}. \label{eq:product}
\end{equation*}
%where $(a_1,a_2,a_3)$ is the side-length of $\Delta$.
\end{lemma}

\begin{proof}
Repeating use of  \eqref{eq:quotient1} and Lemma \ref{lem:Rauch} applied to $s=1/2$ implies that 
for each $I=i_1\cdots i_m$,
\begin{align*}
  (1-r^2_m)\cdots (1-r_1^2)&(1-r^2) \frac{a_i}{a_j} \\
                                     &<\frac{a_{I,i}}{a_{I,j}} < (1+r^2_m)\cdots (1+r_1^2)(1+r^2) \frac{a_i}{a_j} 
\end{align*}  
for every $1\le i,j \le 3$,  where $r_k:=|\Delta_{i_1\cdots i_k}|$, $1\le k\le m$.
Since  
\[
            r_k\le \frac{1}{2}(1+r_{k-1}^2)r_{k-1}<\nu r_{k-1}<\cdots < \nu^k r.
\]
it follows that 
\begin{equation}
       \Pi_{m=0}^{\infty}\left( 1-\nu^{2m}r^2\right )\frac{a_{i}}{a_{j}} <  \frac{a_{I,i}}{a_{I,j}} <
 \Pi_{m=1}^{\infty}\left( 1+\nu^{2m}r^2\right )\frac{a_{i}}{a_{j}}. \label{eq:product}
\end{equation}
This completes the proof.
\end{proof}

From \eqref{eq:product}, one can take $L(r)$ as 
\[
            L(r) :=e^{\frac{2r^2}{1-\nu^2}}.
\]

For every $s\in (0,1]$ we denote by $\Delta(1:s)$ the geodesic triangle
$(\gamma_2|_{[0,s]}, \gamma_3|_{[0,s]}, \sigma_{s})$. Similarly,
$\Delta(i:s)$ and $\Delta_I(i:s)$ are defined for every $1\le i\le 3$
and every multi-index $I\in\mathcal I^*$.

Lemmas \ref{lem:nondeg}, \ref{lem:Rauch} and \ref{lem:quotient} imply 

\begin{lemma} \label{lem:nondeg2}
For every $\delta>0$, there exists a positive number  
$r$  such that 
if $\Delta$ is $\delta$-non-degenerate and the diameter  $|\Delta|$ of
 $\Delta$ is less than $r$, 
then $\Delta_I$ as well as $\Delta_I(i:s)$ is $\delta/2$-non-degenerate
for every multi-index $I$, $1\le i\le 3$ and $s\in (0,1)$. 
\end{lemma}

By Lemma \ref{lem:nondeg2},  we get the conclusion $(1)$ of Theorem \ref{thm:Sierp-small}.
In view of Theorem \ref{T1}, to prove the conclusion $(2)$ of Theorem \ref{thm:Sierp-small},
it suffices to prove the following.

\begin{theorem} \label{thm:alm-similar}
There is a positive numbers $c=c(\delta)$ such that 
%$\lim_{r\to 0} \nu(r)=1/2$  such that 
$\{ (\Delta_I, f_I)\}_{I\in \mathcal I^*}$ gives a $(1/2,\varphi_c, \nu)$-asynptotic similarity system,
where $\varphi_c(x)=cx^2$. 
\end{theorem}

\begin{proof}
In view of Lemma \ref{lem:nondeg2}, it suffices to  prove that 
the map $f:=f_1:\Delta\to\Delta_1\subset\Delta$ is a 
 $(1/2,\varphi_c, \nu)$-almost similarity map for a uniform positive constant $c=c(\delta)$.
Note that $J_s(t):=\frac{\partial \varphi}{\partial s}(t,s)$ is a Jacobi field along
$\sigma_s$. Set $T_s(t):=\frac{\partial \varphi}{\partial t}(t,s)=\dot\sigma_s(t)$.
Observe that
\begin{equation}
    df(T_s(t))=T_{s/2}(t), \,\,\,  df(J_s(t))=\frac{1}{2}J_{s/2}(t). \label{eq:deriv}
\end{equation}
Lemma \ref{lem:Rauch} shows that 
\[
      \left|  \frac{L(\sigma_{s/2})}{L(\sigma_{s})} -\frac{1}{2} \right| < 3r^2,
\]
which implies that
\begin{align}
     \left|  \frac{|df(T_s)|}{|T_s|}  -\frac{1}{2} \right| < 3r^2. \label{eq:ratio1}
\end{align}
Next we show 

\begin{lemma} \label{lem:normal-jacobi}
For every $s,u\in (0,1]$ and $t\in [0,1]$, we have 
\begin{align}
     \left|  \frac{|J_{u}(t)|}{|J_s(t)|}  - 1 \right| < C(\delta)r^2. \label{eq:ratio2}
\end{align}
\end{lemma}

From now on, we shall use the general symbols $C(\delta)$ or $c(\delta)$
to denote constants depending only on $\delta$ unless otherwise stated.

\begin{proof}
For any fixed $s$, take  unique Jacobi fields  $Y_1$ and $Y_2$ along $\sigma_s$
and the reverse geodesic $\sigma_s^{-}(t):=\sigma(1-t)$ respectively such that
\[  
      Y_1(0)=0,  \, Y_1(1) =J_s(1),  
      \,\,  Y_2(1)=J_s(0), \, Y_2(0)=0,
\]
to have  
\[
          J_s(t)=Y_1(t))+Y_2(1-t).
\]
We dente by $\mathbb S^2$ and $\mathbb H^2$ the sphere and the
 hyperbolic plane of constant curvature $1$ and $-1$ respectively. 

Recall that $\Delta$ is a $\delta$-non-degenerate geodesic triangle region of
side lengths $(a_1, a_2, a_3)$ in $D$ whose diameter is denoted by $r$.

\begin{lemma}\label{lem:angle1}
Let $\alpha_{i+}$ and  $\alpha_{i-}$be the angles of comparison
 triangles  $\Delta_{+}$ and $\Delta_{-}$ of $\Delta$
in $\mathbb S^2$ and $\mathbb H^2$ respectively at the vertices  opposite to 
the edge of length $a_i$. Then we have
\[
      |\alpha_{i+} - \alpha_{i-}| <  C(\delta)r^2.
\]
%where $r=|\Delta|$ and $c(\delta)$  is a constant depending only on $\delta$.
\end{lemma}

\begin{proof}
Put $(a,b,c):=(a_1, a_2, a_3)$, and let $\alpha_{+}$,  $\alpha_{-}$ and $\alpha$ be the angles of 
comparison triangles of $\Delta$ in $\mathbb S^2$, $\mathbb H^2$  and 
$\mathbb R^2$ respectively 
at the vertices  opposite to the edge of length $a$.
By the laws of cosines, we have
\begin{align*} %\label{eq:cos1}
       \sin b\sin c\cos\alpha_{+} &=\cos a-\cos b\cos c    \\
  \sinh b\sinh c\cos\alpha_{-} &=\cosh b\cosh c -\cosh a\\
  2bc\cos\alpha &= b^2 + c^2 -a^2,
\end{align*}
which imply 
\begin{align*}
       2bc\cos\alpha_{+} &=  2bc\cos\alpha +O(b^3c)+O(bc^3)+O(b^2c^2) +O(a^4)  \\
       2bc\cos\alpha_{-} &= 2bc\cos\alpha +O(b^3c)+O(bc^3)+O(b^2c^2) +O(a^4). %\label{eq:cos2}
\end{align*}
It follows from Sublemma\ref{slem:edge-quot} that
\begin{align*}
       |\cos\alpha_{+} -   \cos\alpha| &\le O(b^2)+O(c^2)+O(bc) +O(a^4/bc)  \\
                                      & \le C(\delta)r^2.
\end{align*}
Since $\delta<\alpha<\pi-\delta$, we obtain
$|\alpha_{+}-\alpha|\le C(\delta)r^2$.
Similarly we get $|\alpha_{-}-\alpha|\le C(\delta)r^2$, and hence 
$|\alpha_{+}-\alpha_{-}|\le C(\delta)r^2$.
\end{proof}

Let $\alpha_{s}$ and $\beta_{s}$ be the angle of the geodesic triangle 
$\Delta(1:s) =(\gamma_2|_{0,s]}, \gamma_3|_{[0,s]}, \sigma_s)$ at $\gamma_2(s)$ and $\gamma_3(s)$
respectively.

\begin{lemma}\label{lem:angle2}
\[
      |\alpha_{s} - \alpha_{t}| <  c(\delta)r^2, \,\, |\beta_{s} - \beta_{t}| <  c(\delta)r^2,
\]
for every  $s,t\in (0,1]$.
\end{lemma}

\begin{proof}
Let $\alpha_{s}^{+}$ , $\alpha_{s}^{-}$ , $\alpha_{s}^{0}$ denote the
 angles of comparison triangles in $\mathbb S^2$, $\mathbb H^2$, and
 $\mathbb R^2$ respectively at the vertices coresponding $\gamma_2(s)$.
By Toponogov's theorem (cf. \cite{CHE}), we have
\begin{align}
       \alpha_s^{-} \le  \alpha_s,  \, \alpha_s^{0} \le  \alpha_s^{+}. \label{eq:Toponogov}
\end{align}
By the law of cosines, we have
\begin{align*}
   \cos\alpha_s^0 &=\frac{a_2^2+(a_1(s)/s)^2-a_3^2}{2a_2(a_1(s)/s)} \\
    \cos\alpha_t^0 &=\frac{a_2^2+(a_1(t)/t)^2-a_3^2}{2a_2(a_1(t)/t)},
\end{align*}
which imply with Lemma\ref{lem:Rauch}
\begin{align*}
   \cos\alpha_s^0  & -  \cos\alpha_t^0  \\
    & \le \frac{a_2^2+a_1^2(1+r^2)-a_3^2}{2a_2a_1(1-r^2)} -\frac{a_2^2+a_1^2(1-r^2)-a_3^2}{2a_2a_1(1+r^2)}    \\
    &=\frac{r^2(2a_1^2+a_2^2-a_3^2)}{a_1a_2(1-r^2)(1+r^2)}\\
    &=\frac{r^2}{1-r^4}\left(\frac{2a_1}{a_2}+\frac{a_2}{a_1}-\frac{a_3^2}{a_1a_2} \right)\\
    &\le C(\delta)r^2.
\end{align*}
Revercing the role of $s$ and $t$, we have 
\[
   |\cos\alpha_s^0 -  \cos\alpha_t^0|\le C(\delta)r^2.
\]
By Lemma \ref{lem:nondeg2}, we have $\delta/2 < (\alpha^0_s+\alpha_t^0)/2 < \pi -\delta/2$,
which implies  $\sin \frac{\alpha^0_s+\alpha_t^0}{2} > \sin (\delta/2)$.
Therefore we conclude that 
\[
     |\alpha_s^0 - \alpha_t^0| \le 4\left| \sin\left( \frac{\alpha_s^0
 -\alpha_t^0}{2}\right)\right| \le C_1(\delta)r^2. 
\]
where $C_1(\delta):= \frac{2C(\delta)}{\sin (\delta/2)}$
Using \eqref{eq:Toponogov} and Lemma \ref{lem:angle1}, we see
\begin{align*}
      \alpha_s & \le \alpha _s^0 + C(\delta)r^2  \\
                  &\le \alpha _t^0 + C(\delta)r^2 +C_1(\delta)r^2\\
                  & \le \alpha_t + 2C(\delta)r^2 +C_1(\delta)r^2.
\end{align*}
Reversing the role of $s$ and $t$ completes the proof.
\end{proof}

Next we analyze the behavior of the norm of Jacobi field $J_s$. 
For a fixed $s\in (0,1]$, 
let $Y_i(t)=Y_i^{N}(t)+Y_i^{T}(t)$, $i=1,2$, be the orthogonal decompositions
of $Y_i$ to the normal and tangential components to $\dot\sigma_s$. 
We can write $Y_i(t)$ and $Y_i(t)^N$ as 
\begin{align} \label{eq:express2}
      Y_1(t)&=d \exp_{\gamma_2(s)} (t(V_1)_{t\dot\sigma_s(0)}),\,\,
     Y_2(t)=d \exp_{\gamma_3(s)} (t(V_2)_{t\dot\sigma_s^{-}(0)}), \\
                            \label{eq:express1}
      Y_1^{N}(t)&=d \exp_{\gamma_2(s)}( t(V_1^N)_{t\dot\sigma_s(0)}),\,\,
     Y_2^{N}(t)=d \exp_{\gamma_3(s)} (t(V_2^N)_{t\dot\sigma_s^{-}(0)}),
\end{align}
where $V_1$ and $V_2$ are some parallel vector fields on the tangent spaces
satisfying
\[
  d \exp_{\gamma_2(s)} ((V_1)_{\dot\sigma_s(0)}) =\dot\gamma_3(s),\,\,
  d \exp_{\gamma_3(s)} ((V_2)_{\dot\sigma_s^{-}(0)}) =\dot\gamma_2(s).
\]
The Rauch comparison theorem shows that
\begin{equation*}
      |Y_1^{N}(t)| \fallingdotseq t|V_1^N|  \fallingdotseq t|\dot\gamma_3(t)^N|, \,\,
     |Y_2^{N}(1-t)|\fallingdotseq (1-t)|V_2^N|  \fallingdotseq (1-t)|\dot\gamma_2(t)^N|.
\end{equation*}
Here and hereafter  we use the symbol $a\fallingdotseq b$ whenever  
$\left|\frac{a}{b}-1\right|<C(\delta)r^2$.
It follows from $\dim M=2$ that
\begin{align}
|J_s^N(t)| &= |Y_1^{N}(t)| +  |Y_2^{N}(1-t)| \\
             & \fallingdotseq t|\dot\gamma_3(t)^N|+    (1-t)|\dot\gamma_2(t)^N| \\
             &=t \sin \beta_s a_3 + (1-t)\sin\alpha_s a_2,
                           \label{eq:angle2}
\end{align}
where we recall  $a_i=L(\gamma_i)=|\dot\gamma_i(t))|$. 
Similarly we have
\begin{equation*}
    |J_{u}^N(t)| \fallingdotseq 
                    t \sin \beta_{u}a_3+ (1-t)\sin\alpha_{u}a_2.
                                \label{eq:angle3}
\end{equation*}
It follows from %\eqref{eq:angle2}, \eqref{eq:angle3} 
that
\begin{equation}
      |J_s^N(t)|   \fallingdotseq |J_{u}^N(t)|. \label{eq:norm}
\end{equation}
Next we show that
\begin{equation}
      |J_s^T(t)|   \fallingdotseq  |J_{u}^T(t)|. \label{eq:tangl}
\end{equation}
We use the expression \eqref{eq:express2} with Gauss's lemma to obtain 

\begin{align*} \label{eq:tang2}
      & \langle Y_1(t), T_s(t)  \rangle = t a_3|T_s|\cos\beta_s,\,\, \\
       &\langle Y_2(t), T_s(t)  \rangle  =  -(1-t) a_2|T_s|\cos\alpha_s.
\end{align*}
Thus we get
\[
       |J_s^T(t)|  =  |t a_3\cos\beta_s -(1-t)a_2\cos\alpha_s|.
\]
From an inequality for $|J_u^T(t)|$ similar to the above and Lemma \ref{lem:angle2}, we have
\eqref{eq:tangl}.
Now \eqref{eq:ratio2} follows from \eqref{eq:norm}, \eqref{eq:tangl}.
Thus we have completed the proof of Lemma \ref{lem:normal-jacobi}.
\end{proof}

The expression \eqref{eq:express2} also yields 
\[
    |Y_1(t)|\fallingdotseq t|V_1| \fallingdotseq t a_3, \,\,
     |Y_2(1-t)|\fallingdotseq (1-t)|V_2| \fallingdotseq (1-t) a_2.
\]
In particular we have 
\begin{equation} \label{eq:jacobi}
         |J_s(t)|\le 2r.
\end{equation}
Since  $|J_s^N(t)|\ge c(\delta)r$ from \eqref{eq:angle2},  \eqref{eq:jacobi}
implies that the angle $\theta_s(t):=\angle(J_s(t), T_s(t))$ has definite
lower and upper bounds:
\begin{equation}
          0<c(\delta) \le  \theta_s(t) \le \pi - c(\delta).   \label{eq:angle4}
\end{equation}
\eqref{eq:deriv}, 
\eqref{eq:ratio1}, \eqref{eq:ratio2} and \eqref{eq:angle4} yield that 
\[
     \left| \frac{|df(v)|}{|v|} - \frac{1}{2} \right| < C(\delta)r^2,
\]
for every tangent vector $v$. 
Thus we conclude that
$f:\Delta\to\Delta_1$ is a $(1/2, \varphi_{C(\delta)}, \nu)$-almost similarity map,
with $\varphi_{C(\delta)}(x)=C(\delta)x^2$.
This completes the proof of Theorem $(2)$ \ref{thm:alm-similar}.
\end{proof}

\begin{proof}[Proof of Corollary \ref{thm:nondeg}]
In view of Theorem \ref{T1}, it suffices to show that 
for a geodesic triangle region  $\Delta$ on a convex domain of 
a complete surface, if the collection $\{ (\Delta_I,f_I)\}_{I\in \mathcal I^*}$ 
gives a $(\{ 1/2,1/2,1/2\},\varphi_C, \nu)$-asymptotic similarity system
with  $\varphi_C(x)=Cx^2$ and $0<\nu < 1$,
then $\Delta$ is asymptotically non-degenerate.

For a large $n_0$, fix an abitrary $I_0=i_1\cdots i_{n_0}\in\mathcal I_{n_0}$, 
and set  
\[
 W:= \Delta_{I_0}= g_{I_0}(\Delta)=f_{I_0}\circ \cdots f_{i_1i_2}\circ
 f_{i_1}(\Delta).
\]
For every $1\le i\le k$, put 
\[
     h_i:=f_{I_0i}:W\to  W_{i}=h_i(W)\subset  W, 
\]
and recall from the definition
\[
       \left| \frac{d(h_i(x), h_i(y))}{d(x,y)} - \lambda_i\right| < o(n_0),
\]
where $o(n_0)=\lambda_i\varphi(\nu^{n_0}|\Delta|)$ and therefore
$\lim_{n_0\to\infty} o(n_0)=0$. 
For $J=j_1\cdot \cdot j_m$, define 
$g_J:W \to W_J$ by 
\[
                           g_J:= h_J\circ \cdots \circ h_{j_1j_2}\circ h_{j_1},
\]
where we use the notation 
\[
 h_{j_1\cdot \cdot j_{\ell}}:= f_{Ij_1\cdot \cdot j_{\ell}}
: W_{j_1\cdot \cdot j_{\ell-1}} \to  W_{j_1\cdot \cdot j_{\ell}},
\]
as before. By Lemma \ref{lem:g-dil}, 
we have 
\[ 
      \left| \frac{d(g_J(x), g_J(y))}{d(x,y)} - \lambda_J\right| < o(n_0)\lambda_J,
\]
for every $x,y\in W$. We denote by ${\rm inrad}(W)$, the inradius of $W$, the largest $r>0$ such 
that an $r$-ball is contained in $W$.
It follows that 
\[
    \frac{|W_J|}{{\rm inrad}(W_J)}\le \frac{1+o(n_0)}{1-o(n_0)} \frac{|W|}{{\rm inrad}(W)},
\]
for every $J\in\mathcal I^*$. This implies that 
there exists a $\delta>0$ such that $\Delta_I$ is 
$\delta$-nondegenerate for every $I\in\mathcal I^*$.
\end{proof}

% \begin{theorem} 
% A geodesic triangle domain $\Delta$ in a convex domain on a complete surface is 
% asymptotically non-degenerate if and only if 
% there exist constants $c>0$ and $0<\nu<1$ such that 
% in the asymptotic similarity system $\{ \Delta_I,f_I\}_{I\in \mathcal I^3}$,
% each $f_I:\Delta_{I'}\to \Delta_I$ is a $(1/2,c,\nu)$-almost similarity map. 
% \end{theorem}

% If a geodesic triangle domain $\Delta$ in a convex domain on a complete
% surface provides an asymptotic similarity system $\{ \Delta_I,f_I\}_{I\in \mathcal I^3}\}$
% on $\Delta$, each $f_I:\Delta_{I'}\to \Delta_I$ is a $(1/2,c, \nu)$-almost similarity map
% for some uniform constants $c>0$ and $0<\nu<1$, then shows that $\Delta$ is  asymptotically non-degenerate.
% Thus we have just  proved 

% \begin{corollary} \label{cor:Sierp}
% If a geodesic triangle domain $\Delta$ in a convex domain on a surface is 
% asymptotically non-degenerate, then an asymptotic Sierpinski gasket $K$ on
% $\Delta$ can be constructed,  and its Hausdorff and box dimensions are  given
% by
% \begin{equation}
%    \dim_H K = \dim_B K = \frac{\log3}{\log2}. \label{eq:HB-dim}
% \end{equation}
% \end{corollary}

%%%%start of the bibliography

\end{document}